\documentclass{article}

\usepackage[english]{babel}

\usepackage[letterpaper,top=2cm,bottom=2cm,left=3cm,right=3cm,marginparwidth=1.75cm]{geometry}

\usepackage{float}
\usepackage{amsmath,amssymb,amsthm,tikz-cd,xspace,subfigure}
\usetikzlibrary{positioning, shapes.geometric, calc, intersections}
\tikzstyle{vertex}=[circle,fill=black,inner sep=2pt]
\tikzstyle{vertrect}=[draw,rectangle,inner sep=2pt]
\tikzstyle{vertdia}=[draw,diamond,inner sep=2pt]
\usepackage{amsmath}
\usepackage{graphicx}
\usepackage[colorlinks=true, allcolors=blue]{hyperref}
\usepackage{amsmath, amssymb, amsthm, amsfonts}
\usepackage{fullpage}
\usepackage{hyperref}
\usepackage{enumitem}
\usepackage{mathtools}
\usepackage[capitalise]{cleveref}

\theoremstyle{definition}

\theoremstyle{theorem}

\newtheorem{thm}{Theorem}[section]

\newtheorem{lemma}[thm]{Lemma}
\newtheorem{obs}[thm]{Observation}

\theoremstyle{remark}

\newcommand{\conv}{{\rm{conv}}}

\def\conv{\mbox{\rm conv}}

\title{Big line or big convex polygon}
\author{David Conlon\thanks{Department of Mathematics, California Institute of Technology, Pasadena, CA 91125. Email: dconlon@caltech.edu. Research supported by NSF Awards DMS-2054452 and DMS-2348859.} \and
Jacob Fox\thanks{Department of Mathematics, Stanford University, Stanford, CA 94305. Email: jacobfox@stanford.edu. Research supported by NSF Award DMS-2154129.} \and
Xiaoyu He\thanks{Department of Mathematics, Princeton University, Princeton, NJ 08544. Email: xiaoyuh@princeton.edu. Research supported by NSF Award DMS-2103154.} \and
Dhruv Mubayi\thanks{Department of Mathematics, Statistics and Computer Science, University of Illinois, Chicago, IL 60607. Email: mubayi@uic.edu. Research partially supported by NSF Awards
DMS-1952767 and DMS-2153576.} \and
Andrew Suk\thanks{Department of Mathematics, University of California at San Diego, La Jolla, CA 92093. Email: asuk@ucsd.edu. Research supported by an NSF
CAREER Award and by NSF Awards DMS-1952786 and DMS-2246847.} \and 
Jacques Verstra\"ete\thanks{Department of Mathematics, University of California at San Diego, La Jolla, CA 92093. Email: jacques@ucsd.edu. Research supported by NSF Award DMS-1800332.}}
\date{}

\begin{document}

\maketitle

\begin{abstract}
   Let $ES_{\ell}(n)$ be the minimum $N$ such that every $N$-element point set in the plane contains either $\ell$ collinear members or $n$ points in convex position.  We prove that there is a constant $C>0$ such that, for each $\ell, n \ge 3$, 
   $$ (3\ell - 1) \cdot 2^{n-5}  < ES_{\ell}(n) < \ell^2 \cdot 2^{n+ C\sqrt{n\log n}}.$$ A similar extension of the well-known Erd\H os--Szekeres cups-caps theorem is also proved.
\end{abstract}

\section{Introduction}

Given an $n$-element point set $P$ in the plane, we say that $P$ is in \emph{convex position} if $P$ is the vertex set of a convex $n$-gon.  We say that  $P$ is in \emph{general position} if no three members of $P$ are collinear.  In  1935, addressing a problem raised by Klein, Erd\H os and Szekeres \cite{ES35} proved that, for every integer $n\geq 3$, there is a minimal integer $ES(n)$ such that any set of $ES(n)$ points in the plane in general position contains $n$ members in convex position.  Moreover, they showed that $ES(n)\le {2n-4 \choose n-2} + 1=4^{n+o(n)}$. Many years later~\cite{es2}, they proved that $ES(n) \geq 2^{n-2} + 1$, a bound that they had already conjectured to be tight in their earlier paper. It remained an open problem for several decades to improve the bound $ES(n) \le 4^{n+o(n)}$ by any significant factor. 
This was finally accomplished by Suk~\cite{S}, who proved that $ES(n) = 2^{n+o(n)}$, coming close to matching Erd\H{o}s and Szekeres' lower bound and proving their conjecture.  The best explicit bound for the $o(n)$ term to date is due to Holmsen et al.~\cite{holm}, who optimized the argument in \cite{S} and showed that $ES(n)\le 2^{n + O(\sqrt{n\log n})}$.

In this paper, we extend these results to arbitrary point sets in the plane.  Let $ES_{\ell}(n)$ be the minimum $N$ such that every $N$-point set in the plane contains either $\ell$ collinear points or $n$ points in convex position. Hence, $ES_3(n) = ES(n)$. For $\ell \geq 3$, we prove the following.

\begin{thm} \label{convex}
    There exists $C>0$ such that, for each $\ell, n \ge 3$,  $ES_{\ell}(n)\le 
  \ell^2 \cdot 2^{n+C\sqrt{n\log n}}$. 
\end{thm}

  The proof of Theorem \ref{convex} is based on both a new cups-caps theorem for arbitrary point sets in the plane and a new positive fraction Erd\H os--Szekeres theorem. In the case where $n$ is fixed and $\ell$ tends to infinity, our cups-caps theorem implies that $ES_{\ell}(n) = O(\ell)$, which is best possible up to constants. 
  In turn, our lower bound for the cups-caps theorem implies the following lower bound for $ES_{\ell}(n)$, which agrees with the Erd\H os--Szekeres lower bound in the $\ell = 3$ case.

\begin{thm}\label{thmlower}
  For each $\ell, n \ge 3$, $ES_{\ell}(n)\geq 
   (3\ell - 1) \cdot 2^{n-5}  + 1$. 
\end{thm}

It remains an interesting open problem to determine the correct dependence of $ES_{\ell}(n)$ on $\ell$. 

The paper is organized as follows.  In the next section, we prove our cups-caps theorem for arbitrary point sets and Theorem \ref{thmlower}.  In Section \ref{secpos}, we establish a positive fraction Erd\H os--Szekeres theorem for arbitrary point sets.  Finally, in Section \ref{secpf}, we prove Theorem \ref{convex}.  For the sake of clarity, we omit floor and ceiling signs whenever they are not crucial. We assume throughout that our point sets have distinct $x$-coordinates, since we can slightly rotate the plane otherwise.

\section{A cups-caps theorem for arbitrary point sets}\label{secaux}

  Let $X$ be a $k$-element point set in the plane with distinct $x$-coordinates. We say that $X$ forms a \emph{$k$-cup} (\emph{$k$-cap}) if $X$ is in convex position and its convex hull is bounded above (below) by a single edge.  The \emph{length} of a $k$-cup ($k$-cap) is $k-1$. 
 Write $f_{\ell}(m,n)$ for the  minimum $N$ such that every $N$-point set in the plane contains either $\ell$ collinear members, an $m$-cup or an $n$-cap. Erd\H os and Szekeres~\cite{ES35} proved that 
  
  \begin{equation}\label{escupcap}f_{3}(m,n) ={m+n-4 \choose n-2} + 1.\end{equation} For $\ell \geq 3$, we prove the following.

\begin{thm} \label{cupscaps}
There is an absolute constant $c > 1$ such that, for $m, n, \ell \ge 3$, 
$$f_{\ell}(m,n)\le  c(\min\{m-1,n-1\}  + \ell)\cdot {m+n-4 \choose n-2}.$$  
\end{thm}

The proof of this theorem is based on a connection between down-sets and (\ref{escupcap}) discovered by Moshkovitz and Shapira \cite{MS}.   We will also need the following lemma due to Beck~\cite{beck}.

\begin{lemma}[Theorem 1.2 in \cite{beck}]\label{beck}
There is an absolute constant $\varepsilon> 0 $ such that every $t$-element point set in the plane contains either $\varepsilon t$ collinear points or determines at least $\varepsilon\binom{t}{2}$ distinct lines.  
    
\end{lemma}

\begin{proof}[Proof of Theorem \ref{cupscaps}]
  Let $\varepsilon > 0$ be the absolute constant from Lemma \ref{beck} and set $c = 10/\varepsilon$.  Let $P$ be an $N$-element point set in the plane where $N  =   c\cdot (\min\{m-1,n-1\}  + \ell)\cdot {m+n-4 \choose n-2} + 1 . $  We may assume that $P$ does not contain $\ell$ collinear members, since otherwise we would be done.   Given points $p,q \in P$, we write $p < q$ if the $x$-coordinate of $p$ is less than the $x$-coordinate of $q$.   For the sake of contradiction, suppose $P$ contains neither an $m$-cup nor an $n$-cap.  Hence, the longest cup in $P$ has length at most $m-2$ and the longest cap in $P$ has length at most $n -2$. 

Let $p,q \in P$ be such that $p < q$.  We label the pair $pq$ with the ordered pair $(x_{pq},y_{pq})$, where $x_{pq}$ is the length of the longest cup ending at $pq$ and $y_{pq}$ is the length of the longest cap ending at $pq$. Let $L(m-2, n-2)$ be the poset on $[m-2] \times [n-2]$ where $(x,y) \preceq (x',y')$ iff $x\le x'$ and $y\le y'$.  For each $q \in P$, let $S(q)=\{(x_{pq}, y_{pq}): p \in P, \, p < q\}$.  Let $D(q)=\{(x,y) \in L(m-2,n-2): \exists (x_{pq}, y_{pq}) \in S(q), \, (x,y)\preceq (x_{pq},y_{pq})\}$ be the \emph{down-set} in $L(m-2,n-2)$ generated by $S(q)$.  

 The number of down-sets in $L(m-2, n-2)$ is ${m+n-4 \choose n-2}$ (see, e.g., \cite[Observation~2.1]{MS}). Hence, by the pigeonhole principle, there are points $q_1<q_2< \cdots < q_t$ in $P$ with $t\geq c\cdot (\min\{m-1,n-1\} + \ell)$ such that $D(q_i)=D(q_j)$ for all $i<j$.  Set $Q  = \{q_1,\ldots, q_t\}$.  By Lemma~\ref{beck}, $Q$ contains either $\varepsilon t$ collinear members or determines at least $\varepsilon\binom{t}{2}$ distinct lines.  In the former case, we have $\varepsilon t > \ell$ collinear points, which is a contradiction.  Hence, $Q$ determines at least $\varepsilon\binom{t}{2}$ distinct lines.  By averaging, there is a point $p \in Q$ and a subset $Q'\subset Q$ of size at least $\varepsilon t/2 > \min\{m-1,n-1\}$ such that $p < q$ for each $q \in Q'$ and there are $|Q'|$ distinct lines between $p$ and $Q'$.  Consider the labels on $pq$ for each $q \in Q'$.  Since the maximum size of an antichain in $L(m-2,n-2)$ is $\min\{m-1,n-1\}$, by the pigeonhole principle, we obtain points $p, q, q'$ such that $p<q<q'$ and

\begin{enumerate}
    \item $D(p)=D(q)=D(q')$ and 

    \item $(x_{pq}, y_{pq}) \preceq (x_{pq'}, y_{pq'})$ or $(x_{pq}, y_{pq}) \succeq (x_{pq'}, y_{pq'}).$
\end{enumerate}

Let us assume by symmetry that $(x_{pq}, y_{pq}) \succeq (x_{pq'}, y_{pq'})$. Since $D(p)=D(q)$, there exists $(x,y) \in S(p) \subset D(p)$ such that $(x,y) \succeq (x_{pq}, y_{pq})$ and, by transitivity, $(x,y) \succeq (x_{pq'}, y_{pq'})$. By the definition of $S(p)$, there exists $p'<p$ such that $x=x_{p'p}$ and $y=y_{p'p}$. 
Since $p, q, q'$ are not collinear, one of $p'pq, p'pq'$ is not collinear. Without loss of generality, we can assume that $p'pq$ is not collinear, since the other case is symmetric.  Then the triple $p'pq$ is either a cup or a cap.  In the former case, the longest cup ending at $p'p$ with length $x_{p'p}$ can be extended to end at $pq$, which is a contradiction.  If instead $p'pq$ is a cap, then the longest cap ending at $p'p$ with length $y_{p'p}$ can be extended to end at $pq$, again a contradiction.    
\end{proof}

In the other direction, we prove the following.

\begin{thm} \label{thm:capcuplower}
    $$f_{\ell}(m,n) > \frac{\ell - 1}{2}\binom{m + n - 4}{n - 2} - \frac{\ell - 3}{2}\binom{m + n - 6}{n-3}.$$
\end{thm}

\begin{proof}
Set 
$$h_{\ell}(m,n) := \frac{\ell - 1}{2}\binom{m + n - 4}{n - 2} - \frac{\ell - 3}{2}\binom{m + n - 6}{n-3}.$$
In what follows, we will recursively construct planar point sets $X_{\ell, m,n}$ with $|X_{\ell, m,n}| \geq h_{\ell}(m,n)$ that contain neither $\ell$ collinear points, $m$-cups nor $n$-caps.  
For $\ell, m\geq 3$, we construct $X_{\ell, m,3}$ by taking the lower half of a regular $m$-gon and, on $\lfloor (m-1)/2\rfloor$ of these segments, placing $\ell-1$ collinear points in the interior of the segment. If $m-1$ is odd, then add another point on a segment by itself (adding more than one point to this segment would create an $m$-cup). 
Hence, we have no $\ell$ collinear points, no $m$-cup and no $3$-cap.  Moreover, 
\begin{equation*}
    |X_{\ell, m,3}| = \left\{\begin{array}{ll}
 (\ell - 1)\frac{m-1}{2}   & \textnormal{if $m-1$ is even} \\\\
  (\ell-1)\frac{m-2}{2} + 1  & \textnormal{if $m-1$ is odd.}
\end{array}\right.
\end{equation*}
Hence, for all $m \geq 3$, 
$$    |X_{\ell, m,3}| \geq \frac{\ell  - 1}{2}(m-1) - \frac{\ell -3}{2} = h_{\ell}(m,3),$$
as desired.  We construct $X_{\ell, 3,n}$ similarly such that
$$    |X_{\ell, 3,n}| \geq \frac{\ell  - 1}{2}(n-1) - \frac{\ell -3}{2} = h_{\ell}(3,n).$$
For the recursive step, assume that we have constructed $X_{\ell, m',n'}$ for all $m' < m$ or $n' < n$.  We construct $X_{\ell, m, n}$ as follows. Take a very flat copy of $X_{\ell, m-1, n}$ and a very flat copy of $X_{\ell, m,n-1}$ such that $X_{\ell, m,n-1}$ is very high and far to the right of $X_{\ell, m-1, n}$, the line spanned by any two points in $X_{\ell, m-1, n}$ lies below $X_{\ell, m, n-1}$ and the line spanned by any two points in $X_{\ell, m, n-1}$ lies above $X_{\ell, m-1, n}$.  See Figure \ref{figlower2}.

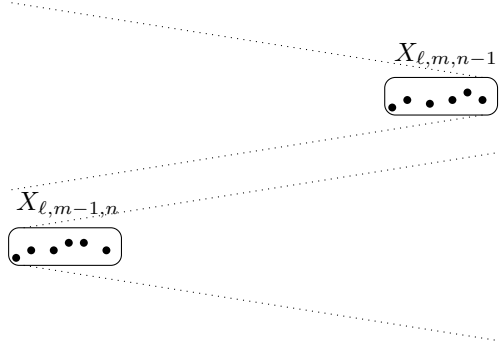
\begin{figure}
\begin{center}
\begin{tikzpicture}
\foreach \x/\y in {.1/0.1, 0.3/0.2, 0.6/0.2, 0.8/0.3, 1/0.3, 1.3/0.2} {
    \fill (\x,\y) circle[radius=1.5pt];
}
\node[below left] at (1.6,1.1) {$X_{\ell, m-1, n}$};
\draw[rounded corners] (0,0) rectangle (1.5,0.5);
\draw[dotted] (0.2,0) -- (6.5,-1) (0.2,0.5) -- (6.5,1.5);

\begin{scope}[xshift=5cm,yshift=2cm]
\foreach \x/\y in {0.1/0.1, 0.3/0.2, 0.6/0.15, 0.9/0.2, 1.1/0.3, 1.3/0.2} {
    \fill (\x,\y) circle[radius=1.5pt];
}
\node[above right] at (0,0.5) {$X_{\ell, m, n-1}$};
\draw[rounded corners] (0,0) rectangle (1.5, .5);
\draw[dotted] (1.3,0) -- (-5,-1) (1.3,0.5) -- (-5,1.5);
\end{scope}

\end{tikzpicture}
  \caption{Construction for $X_{\ell, m, n}$ from $X_{\ell, m,n-1}$ and $X_{\ell, m,n-1}$.}\label{figlower2}
 \end{center}
\end{figure}

Hence, the resulting set does not contain $\ell$ collinear points and neither an $m$-cup nor an $n$-cap.  Finally,
\begin{eqnarray*}
    |X_{\ell, m,n}|& \geq & |X_{\ell,m-1,n}| + |X_{\ell,m,n-1}|\\\\
     & \geq & h_{\ell}(m-1,n) + h_{\ell}(m,n-1)\\\\
     & \geq & \frac{\ell - 1}{2}\binom{m + n - 5}{n - 2} - \frac{\ell - 3}{2}\binom{m + n - 7}{n-3} +  \frac{\ell - 1}{2}\binom{m + n - 5}{n - 3} - \frac{\ell - 3}{2}\binom{m + n - 7}{n-4}\\\\
     & = & \frac{\ell - 1}{2}\binom{m + n - 4}{n - 2} - \frac{\ell - 3}{2}\binom{m + n - 6}{n-3},
\end{eqnarray*}
as required.
\end{proof}

\subsection{Proof of Theorem \ref{thmlower}}

We now use Theorem~\ref{thm:capcuplower} to prove Theorem \ref{thmlower}, the statement that $ES_{\ell}(n)\le 
  \ell^2 \cdot 2^{n+C\sqrt{n\log n}}$ for all $\ell, n \ge 3$.
  
\begin{proof}[Proof of Theorem \ref{thmlower}]  

Let $X_{\ell, m,n}$ be the point set from the proof of Theorem~\ref{thm:capcuplower} with no $\ell$ collinear points, no $m$-cup and no $n$-cap, recalling that
$$|X_{\ell, m,n}| \geq h_{\ell}(m,n) =  \frac{\ell - 1}{2}\binom{m + n - 4}{n - 2} - \frac{\ell - 3}{2}\binom{m + n - 6}{n-3}.$$
Let $S$ be a unit circle in the plane centered at the origin and consider the arc $\alpha$ along $S$ from $(0,1)$ to $(1,0)$.  Place a very small flat copy of $X_{\ell,n,3}$ near $(0,1)$ and a very small flat copy of $X_{\ell, 3, n}$ near $(1,0)$.  Then evenly spread out very small flat copies of $X_{\ell, n-2,4}, X_{\ell, n-3,5},\ldots, X_{\ell, n - 2-i, 4 + i}, \ldots, X_{\ell, 4,n-2}$ along $\alpha$ from top to bottom, between  $X_{\ell,n,3}$ and $X_{\ell, 3, n}$.  We make each copy flat enough that the line generated by any two points in $X_{\ell, n-2 - i,4 + i}$ lies below $X_{\ell, n-2 -j,4 + j}$ for $j < i$ and lies above $X_{\ell, n-2 - j,4 + j}$ for $j > i$. See Figure \ref{figlower}.
Let $P$ be the final point set. Then 
\begin{eqnarray*}
 |P| & \geq &  \frac{\ell - 1}{2}\left( \binom{ n -1}{1} + \left(\sum\limits_{i = 0}^{n-6}\binom{n - 2}{2 + i}\right) + \binom{n - 1}{n - 2}\right) - \frac{\ell - 1}{3}\left( \binom{ n -3}{0} + \left(\sum\limits_{i = 0}^{n-6}\binom{n - 4}{1 + i}\right) + \binom{n - 3}{n - 3}\right)
   \\\\
  & = &\frac{\ell - 1}{2}  \sum\limits_{i = 0}^{n-2}\binom{n - 2}{i}  - \frac{\ell - 3}{2}   \sum\limits_{i = 0}^{n-4}\binom{n - 4}{ i}   \\\\
  &  = & \frac{\ell - 1}{2} 2^{n-2} - \frac{\ell - 3}{2}2^{n-4}\\\\
  & = & (3\ell-1)2^{n-5}.
\end{eqnarray*}

\begin{figure}
\begin{center}
\includegraphics[width=150pt]{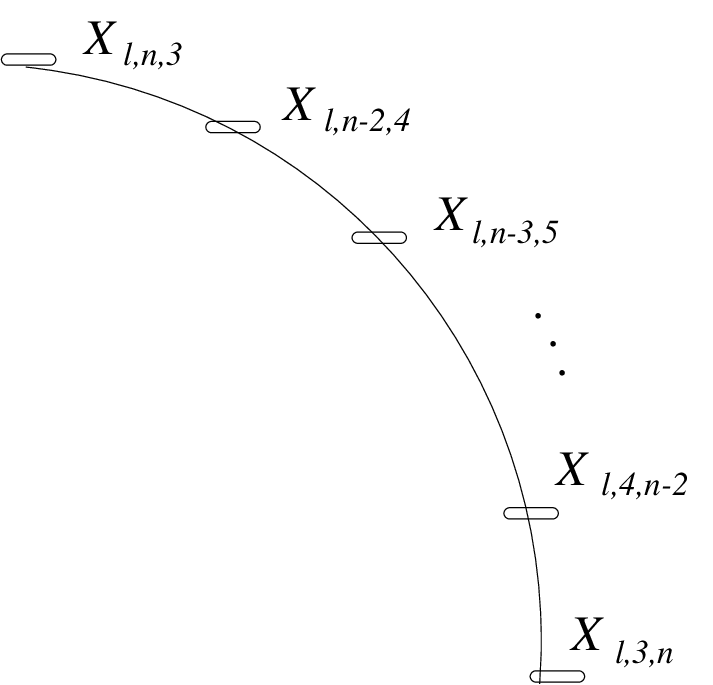}
  \caption{The lower bound construction for $ES_{\ell}(n)$.}\label{figlower}
 \end{center}
\end{figure}

Now suppose that $K \subset P$ is a subset in convex position.  If $K \subset X_{\ell,n-2-i, 4 + i}$ for some $i\ge 0$, then $|K| < n$.  If $K \subset X_{\ell,n,3}$, then $|K| < n$ by the structure of $X_{\ell,n,3}$.  A similar argument holds if $K \subset X_{\ell,3,n}$.

Suppose then that $K$ has a non-empty intersection with at least two of the parts.  Let $i$ be the minimum integer such that $K\cap  X_{\ell, n - 2-i, 4 + i} \neq \emptyset$ and $j$ be the maximum integer such that $K\cap  X_{\ell, n - 2-j, 4 + j} \neq \emptyset$. Assume that $0\le i \le j \le n-6$, that is, that $K$ omits both the highest and lowest sets in our construction.  By the flatness condition, for all $i < s < j$, we have $K\cap X_{\ell, n- 2 -s, 4 + s} \leq 1$.  Hence, 
$$|K| \leq (4 + i - 1) + (j-i-1) + (n - 2 - j -1 )  = n-1.$$

 Suppose now that $|K \cap X_{\ell, n,3}| \ne\emptyset$ and the largest $j$ such that $K\cap  X_{\ell, n - 2-j, 4 + j} \neq \emptyset$ satisfies $0\le j \le n-6$ (or that no such $j$ exists) and
$|K \cap X_{\ell, 3,n}| =\emptyset$.
If $|K \cap X_{\ell, n,3}| \ge 3$, then $K \cap X_{\ell, n,3}$ is a cup, which means that $K \subset  X_{\ell, n,3}$ and hence $|K|\le n-1$. Otherwise, $|K \cap X_{\ell, n,3}| \le 2$ and $|K|\le 2+(n-3-j)+j=n-1$.
A similar argument applies if $|K \cap X_{\ell, 3,n}| \ne\emptyset$.
Finally, if $|K \cap X_{\ell, n,3}|\ne\emptyset$ and
$|K \cap X_{\ell, 3,n}| \ne\emptyset$, then $|K|\le 2+(n-5)+2=n-1$.
Hence, $ES_{\ell}(n) \geq (3\ell-1)2^{n-5} + 1$, as required. 
\end{proof}

\section{A positive fraction cups-caps theorem for arbitrary point sets}\label{secpos}

In this subsection, we establish a positive fraction cups-caps theorem for arbitrary point sets.  Given a $k$-cap ($k$-cup) $X = \{x_1,\ldots, x_{k}\}$, where the points appear in order from left to right, we define the \emph{support of} $X$ to be the collection of open regions $\mathcal{C} = \{T_1,\ldots, T_{k}\}$, where $T_i$ is the region outside of $\conv(X)$ bounded by the segment $\overline{x_ix_{i + 1}}$ and by the lines $x_{i-1}x_{i}$, $x_{i+1}x_{i+2}$ (where $x_{k +1} = x_1$, $x_{k + 2} = x_2$, etc.).   See Figure \ref{figjunctures}.

	\def\junctures{
		\begin{tikzpicture}
		\definecolor{c1}{RGB}{0,0,0}
		\shade [c1] (2,5) --  (2.27,5.85)  -- (4,5) -- cycle;
		\shade [c1] (4,5) --  (5.66,5)  -- (6,4) -- cycle;
		\shade [c1] (6,4) --  (7,1)  -- (8,0.866) -- (8,3) -- cycle;
		\shade [c1] (0.33,0) --  (1,2)  -- (7,1) -- (7.33,0) -- cycle;
		\shade [c1] (0,2.166) -- (1,2) -- (2,5) -- (0,5) -- cycle;
		
		\node at (1,2) [vertex] {};
		\node at (1.3,2.2) {$x_1$};
		\node at (1,3.5) {$T_1$};
		\node at (2,5) [vertex] {};
		\node at (2.2,4.7) {$x_2$};
		\node at (4,5) [vertex] {};
		\node at (4,4.7) {$x_3$};
		\node at (2.7,5.3) {$T_2$};
		\node at (5.2,4.7) {$T_3$};
		\node at (6,4) [vertex] {};
		\node at (5.7,3.8) {$x_4$};
		\node at (7,2.5) {$T_4$};
		\node at (7,1) [vertex] {};
		\node at (6.7,1.3) {$x_5$};
		\node at (4,1) {$T_5$};
		\draw (0.33,0) to  (2.66,7);
		\draw (0,5) to  (8,5);
		\draw (0,7) to  (8,3);
		\draw (5,7) to  (7.33,0);
		\draw (0,2.166) to  (8,0.833);
		\node at (4,3) {$\conv(X)$};
		\end{tikzpicture}	}

		\begin{figure}
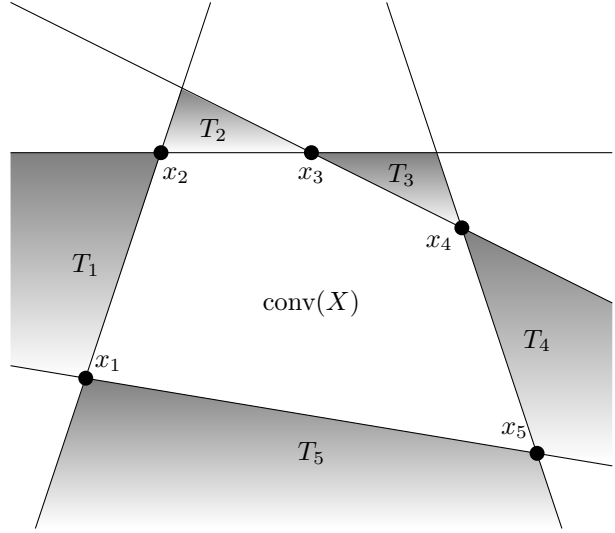
 \centerline{
		{\junctures}}
		\vskip0.2in
		\caption{Regions $T_i$ in the support of $X$.}\label{figjunctures}
	\end{figure}

	\begin{thm}  \label{essupersat}
		
There is a constant $c_1$ such that the following holds. Let $P$ be an $N$-point planar set with no $\ell$ points on a line and $N > c_1\ell\cdot 2^{32k}$. Then there is a $k$-element subset $X\subset P$ that is either a $k$-cup or a $k$-cap such that, for the regions $T_1,\ldots, T_{k-1}$ from the support of $X$, the point sets $P_i =P \cap T_i$  satisfy $|P_i| \geq N/2^{32k}$.   In particular, every $(k-1)$-tuple obtained by selecting one point from each $P_i$, $i = 1,\ldots, k-1$, is in convex position.
		\end{thm}

Let us remark that a positive-fraction cups-caps theorem for point sets in general position was first proved by Pach and Solymosi~\cite{PS} and can be found more explicitly in~\cite{PV}. Its proof is a simple supersaturation argument using (\ref{escupcap}).  Unfortunately, this approach for point sets with no $\ell$ collinear members gives a rather poor dependency on $\ell$.  Instead, we will make use of simplicial partitions together with  the probabilistic method. First, we need some simple definitions. 
A \emph{cell} $\Delta\subset \mathbb{R}^2$ is a 1 or 2-dimensional simplex.  Given a cell $\Delta \subset \mathbb{R}^2$, we say that a line $L$ \emph{crosses} $\Delta$ if $L$ intersects, but does not contain, $\Delta$. 


\begin{lemma}[\cite{mat,chan}]\label{chan}
Let $P$ be a set of $N$ points in the plane.  Then, for any integer $r > 0$, there are disjoint subsets $P_1,\ldots, P_r$ of $P$ and disjoint cells $\Delta_1, \ldots, \Delta_r$ in $\mathbb{R}^2$, with $P_i \subset \Delta_i$, such that $|P_i| \geq N/(8r)$ and every line in the plane crosses at most $O(\sqrt{r})$ cells $\Delta_i$.
\end{lemma}
 
Let us remark that in the original version of simplicial partitions due to Matousek \cite{mat}, the cells $\Delta_i$ may not necessarily be disjoint.  However, in a newer version due to Chan~\cite{chan}, disjointness can also be guaranteed.

	\begin{proof}[Proof of Theorem \ref{essupersat}]
Let $c_1>c_2$ be large constants that will be determined later.  Set $r = c_22^{16k}$. Then we apply Lemma \ref{chan} with parameter $r$ to obtain subsets $P_1, \ldots, P_r \subset P$ and pairwise disjoint cells $\Delta_1,\ldots, \Delta_r \subset \mathbb{R}^2$ such that $|P_i| \geq N/(8r)$ and $P_i\subset \Delta_i$.  Moreover, every line in the plane crosses at most $O(\sqrt{r})$ cells $\Delta_i$.  Since 

$$|P_i|  \geq \frac{N}{8r} \geq \frac{c_1\ell 2^{32k}}{8c_22^{16k}} > \ell,$$

no line contains a cell $\Delta_i$.   We call a triple $(P_i,P_j,P_s)$ of parts \emph{bad} if there is a line intersecting all three cells $\Delta_i,\Delta_j$ and $\Delta_s$.  Otherwise, we call the triple $(P_i,P_j,P_s)$ \emph{good}.  

If there are three disjoint cells $\Delta_i, \Delta_j, \Delta_s$ and a line $L$ that intersects all three, then we can translate and rotate $L$ so that $L$ is tangent to two of the cells and intersects the third.  Hence, for every bad triple $(P_i, P_j, P_s)$, there is a line $L$ tangent to two of the cells, say $\Delta_i$ and $\Delta_j$, such that $L$ intersect $\Delta_s$.  For every pair $\{i, j\}$,
there are at most 4 tangent lines for $\Delta_i$ and $\Delta_j$ and, by our application of Lemma~\ref{chan}, there are at most $O(\sqrt{r})$ parts $\Delta_s$ that intersect any of these 4 lines. Hence, the number of bad triples $(P_i,P_j,P_s)$ is at most $O(r^2\sqrt{r}) = c'r^{5/2}$, where $c'$ is an absolute constant.  

We pick each part $P_i$ with probability $p = 1/(\sqrt{4c'}r^{3/4})$.  Then the expected number of parts chosen is $pr$ and the expected number of bad triples among them is at most
$$p^3 c' r^{5/2} \leq pr/4.$$
Hence, by the Chernoff bound, we can select at least $3pr/4 = \Omega(r^{1/4})$ parts $P_i$ such that the number of bad triples among them is at most $pr/2$. 
By deleting one part from each bad triple, we obtain $pr/4$ parts $P_i$ such that every triple among them is good.  For simplicity, let $P_1,\ldots, P_t$ be the remaining parts, where $t = pr/4 = \Omega(r^{1/4}).$   By sweeping a vertical line from left to right, we can greedily pick subsets $P'_i \subset P_i$, $1 \leq i \leq t$, such that no vertical line intersects any two of the convex sets $C_i = \conv(P'_i)$ and $$|P'_i| \geq |P_i|/t > \Omega(N/r^{5/4}).$$

Without loss of generality, we can assume that the subsets $P'_1,\ldots, P'_t$ appear from left to right.  That is, the $x$-coordinate of each point in $P'_i$ is less than the $x$-coordinate of each point in $P'_j$ for $i < j$.   Let $Q$ be the $t$-element point set obtained by selecting one point from each of the remaining $P'_i$.  Then $Q$ is in general position.  By setting $c_2$ sufficiently large, we have $|Q| = t  = pr/4 \geq 4^{2k}$.  By the Erd\H os--Szekeres cups-caps theorem (\ref{escupcap}), there is either a $(2k)$-cup or a $(2k)$-cap $X\subset Q$.  We will assume that $X$ is a $(2k)$-cap, since a symmetric argument works otherwise. Let $X = \{x_1,\ldots, x_{2k}\}$ be the points of $X$ ordered from left to right and let us now assume that $P_i'$ is the part that corresponds to the point $x_i \in X$.

\begin{obs}
    If $q_1 \in P'_1, \ldots, q_{2k} \in P'_{2k}$, then $q_1,\ldots, q_{2k}$ forms a $(2k)$-cap.  
\end{obs}

\begin{proof}
    It suffices to show that every triple in $\{q_1,\ldots, q_{2k}\}$ forms a cap.  For the sake of contradiction, suppose $(q_i,q_j,q_s)$ is a cup.  Since $(x_i,x_j,x_s)$ is a cap, this implies that the convex sets $\conv(P'_i), \conv(P'_j), \conv(P'_s)$ can be pierced by a line, a contradiction.
\end{proof}

Set $X' = \{x_1,x_3,\ldots, x_{2k-1}\}$.  Let $T_1,\ldots, T_k$ be the support of $X'$.  Then the $k$ parts $P'_{2}, P'_4, \ldots, P'_{2k}$ must lie in $T_1,\ldots, T_k$, respectively.  Moreover, by setting $c_1$ sufficiently large, each such part $P'_{2i}$ satisfies
$$|P'_{2i}| \geq \Omega\left(\frac{N}{r^{5/4}}\right) \geq \frac{N}{2^{32k}},$$
as required. 
\end{proof}

\section{Big line or big convex polygon -- Proof of Theorem~\ref{convex}}\label{secpf}
 For the proof of Theorem \ref{convex}, we will need the following more general version of Theorem \ref{cupscaps}. Let $K$ be a convex set in the plane.  Then we say that the point set $P$ \emph{avoids} $K$ if the line generated by any two points in $P$ is disjoint from $K$.  We say that $K$ and $P$ are $\emph{separated}$ if there is a line that separates $K$ and $\conv(P)$. 
Suppose now that $K$ is a convex set in the plane, $P$ is a finite point set that avoids $K$ and $K$ and $P$ are separated. Then, given a subset $X\subset P$, we say that $X$ is an \emph{inner-cap} with respect to $K$ if, for each point $x \in X$, there is a line that separates $x$ from $(X\setminus\{x\} )\cup K$. Similarly, we say that $X\subset P$ is an \emph{outer-cup} with respect to $K$ if, for each point $x \in X$, there is a line that separates $x\cup K$ from $X\setminus\{x\}$.

\begin{thm} \label{cupscapsC}
There is an absolute constant $c > 0$ such that the following holds.  Let $K$ be a convex set in the plane and let $P$ be a finite point set in the plane that avoids $K$. 
 If $K$ and $P$ are separated and

$$|P| \geq   c(\min\{m-1,n-1\}  + \ell)\cdot {m+n-4 \choose n-2},$$  

then $P$ contains either $\ell$ collinear points, an outer-cup with respect to $K$ of size $m$ or an inner-cap with respect to $K$ of size $n$.  
\end{thm}

\begin{proof}
Let $|P| = N$.  Without loss of generality, we can assume that the line $L$ which separates $K$ and $P$ is horizontal, that $K$ lies below $L$ and that $P$ lies above $L$.   By considering $\conv(K\cup p)$ for each $p \in P$, we can radially order the elements in $P = \{p_1,\ldots, p_N\}$ with respect to $K$ in clockwise order, from left to right.

Notice that every triple in $P$ is either an inner-cap with respect to $K$ or an outer-cup with respect to $K$.  Moreover, for $i < j < s < t$, if $\{p_i,p_j,p_s\}$ and $\{p_j,p_s,p_t\}$ are both inner-caps with respect to $K$ (outer-cups with respect to $K$), then every triple in $\{p_i,p_j,p_s,p_t\}$ is an inner-cap with respect to $K$ (outer-cup with respect to $K$).  Thus, by following the proof of Theorem \ref{cupscaps} almost verbatim, the statement follows.  
\end{proof}

We are now ready to prove Theorem \ref{convex}.

\begin{proof}[Proof of Theorem \ref{convex}]
Let $P$ be an $N$-element planar point set in the plane, where $N =  \ell^2 \cdot 2^{n+ C\sqrt{n \log n}} $ with $C$ a sufficiently large absolute constant.  We can assume that no two points in $P$ have the same $x$-coordinate.  Moreover, we can assume that there are no $\ell$ collinear members in $P$, since otherwise we would be done.

For the sake of contradiction, suppose $P$ does not contain $n$ points in convex position.  Set $k = 2\lceil\sqrt{n\log n}\rceil$.   We apply Theorem~\ref{essupersat} to $P$ with parameter $k  + 3 $, obtaining a subset $X = \{x_1,\ldots, x_{k+3}\}\subset P$ such that $X$ is either a cup or a cap, where we assume that the points of $X$ appear in order from left to right.  Moreover, the regions $T_1,\ldots, T_{k+2}$ in the support of $X$ satisfy
$$|T_i\cap P| \geq \frac{N}{2^{32(k+3)}}.$$   
Set $P_i \subset T_i\cap P$ to be the set of points of $P$ in the interior of $T_i$, for $i = 1,\ldots, k+2$.  Hence,
$$|P_i| \geq  \frac{N}{2^{32(k+3)}} - 3\ell \geq  \frac{N}{2^{40k}}.$$
We will now assume that $X$ is a cap, since a symmetric argument works in the other case.  

Consider the subset $P_i\subset P$ and the region $T_i$ for some fixed $i \in \{2,\ldots, k+1\}$.  Let $B_i$ be the segment $\overline{x_{i-1}x_{i+2}}$.  The point set $P_i$ naturally comes with a partial order $\prec_i$, where $p\prec_i q$ if $p\neq q$ and $p \in \conv(B_i\cup q)$.  Note that $p\prec_i q$ if $p$ lies on the boundary of $\conv(B_i\cup q).$   Following Holmsen et al.~\cite{holm}, for each $P_i$, let

\begin{enumerate}
    \item $h_i$ be the size of the longest antichain with respect to $\prec_i,$

    \item $v_i$ be the size of the longest chain with respect to $\prec_i$, 

    \item $a_i$ be the size of the largest inner-cap with respect to $x_{i+ 1}$ that is also a chain with respect to $\prec_i$,
    
    \item  $b_i$ be the size of the largest inner-cap with respect to $x_{i}$ that is also a chain with respect to $\prec_i$,  
    
    \item $w_i$ be the size of the largest inner-cap with respect to $B_i$ that is also an antichain with respect to $\prec_i$  and
    
    \item $z_i$ be the size of the largest outer-cup with respect to $B_i$ that is also an antichain with respect to $\prec_i$.   
\end{enumerate}

By Dilworth's theorem \cite{dil}, we have $v_ih_i \geq |P_i|$. 
We also clearly have $z_i < n$.   We now make the following observations.

\begin{obs} \label{obs1}
$$w_2 + w_4 + \cdots + w_{k} < n$$

\noindent and $$w_3 + w_5 + \cdots  + w_{k+1} < n.$$
\end{obs}
\begin{proof}
Recall that $k = 2\lceil\sqrt{n\log n}\rceil$ is even.  Let us consider the sets $P_2, P_4, \ldots, P_{k }$. Suppose we have subsets $S_2 \subset P_2, S_4\subset P_4, \ldots, S_{k }\subset P_{k}$ such that $S_i$ is an antichain with respect to $\prec_i$,  an inner-cap with respect to $B_i$ and satisfies $|S_i|= w_i$.  Then we claim that $S = S_{2}\cup S_4 \cup \cdots \cup S_{k}$ is a cap and, therefore, in convex position.  Let $p \in S_{i}$.  Then there is a line $L$ through $p$ that has the property that all of the other points in $S_i$ lie below $L$ and $L$ does not intersect $B_i$.  Since $L$ does not intersect $B_{i}$, all of the points in $S\setminus \{p\}$ must lie below $L$. But then, we must have that
$$w_2 + w_4 + \cdots + w_{k} = |S| < n,$$
as required. 
A similar argument works for the parts $P_3, P_5, \ldots, P_{k+1}$ to prove the second inequality.  
\end{proof}

By Observation~\ref{obs1}, we have

$$w_2 + w_3 + \cdots + w_{k + 1} < 2n.$$

Let $P'_i \subset P_i$ be a chain with respect to $\prec_i$.  Clearly $P'_i$ avoids $x_{i}$ and $x_{i + 1}$.  Moreover, if $P'_i$ contains an outer-cup with respect to $x_i$, then it must be an inner-cap with respect to $x_{i + 1}$.  Therefore, if $|P'_i| > f_{\ell}(m,n)$, then, by Theorem~\ref{cupscapsC} applied to the convex set $K=\{x_{i}\}$, the set $P'_i$ contains either an outer-cup with respect to $x_i$ of size $m$, which is an inner-cap with respect to $x_{i+1}$ of size $m$, or an inner-cap with respect to $x_{i}$ of size $n$.  See Figure \ref{figoutercup}.

\begin{figure}
\begin{center}
\includegraphics[width=150pt]{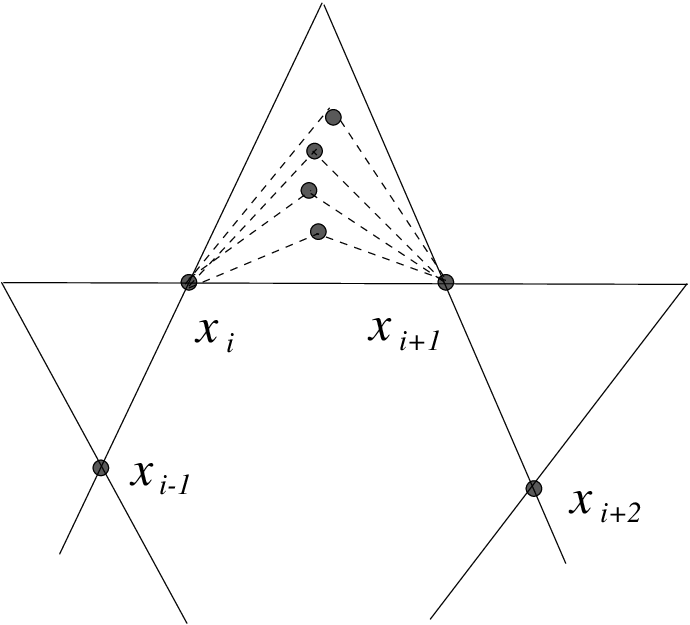}
  \caption{Four points in $P'_i$ that form an outer-cup with respect to $x_i$, which is an inner-cap with respect to $x_{i + 1}$.}\label{figoutercup}
 \end{center}
\end{figure}

\begin{obs} \label{obsadj}
    If there are subsets $Y_{i-1} \subset P_{i-1}$ and $Y_{i} \subset P_{i}$ such that $Y_{i-1}$ is a chain with respect to $\prec_{i-1}$ and an inner-cap with respect to $x_{i}$ and $Y_{i}$ is a chain with respect to $\prec_{i}$ and an inner-cap with respect to $x_{i}$, then $Y_{i-1}\cup Y_{i}$ is in convex position.

\end{obs}

\begin{proof}
It suffices to show that every four points in $Y_{i-1}\cup Y_{i}$ are in convex position.  If all four points lie in $Y_i$, then they are in convex position.  Likewise, if they all lie in $Y_{i -1}$, they are again in convex position.  Suppose we take two points $p_1,p_2 \in Y_{i-1}$ and two points $p_3,p_4 \in Y_{i}$.  Since $Y_{i-1}$ and $Y_{i}$ are both chains with respect to $\prec_{i-1}$ and $\prec_{i}$ respectively, the line spanned by $p_1,p_2$ does not intersect the region $T_{i}$ and the line spanned by $p_3,p_4$ does not intersect the region $T_{i-1}$.  Hence, $p_1,p_2,p_3,p_4$ are in convex position.  Now suppose we have $p_1,p_2,p_3 \in Y_{i-1}$ and $p_4 \in Y_{i}$.  Since the three lines $L_1,L_2,L_3$ spanned by $p_1,p_2,p_3$ all intersect the segment $B_{i-1}$, both $x_{i}$ and $p_4$ lie in the same region in the arrangement of $L_1\cup L_2\cup L_3$.  Therefore, $p_1,p_2,p_3,p_4$ are in convex position. The same argument works for the case where $p_1 \in Y_{i-1}$ and $p_2,p_3,p_4 \in Y_{i}$.  See Figure~\ref{gluepic}.
\end{proof}

\begin{figure}\label{gluepic}
\begin{center}
\includegraphics[width=200pt]{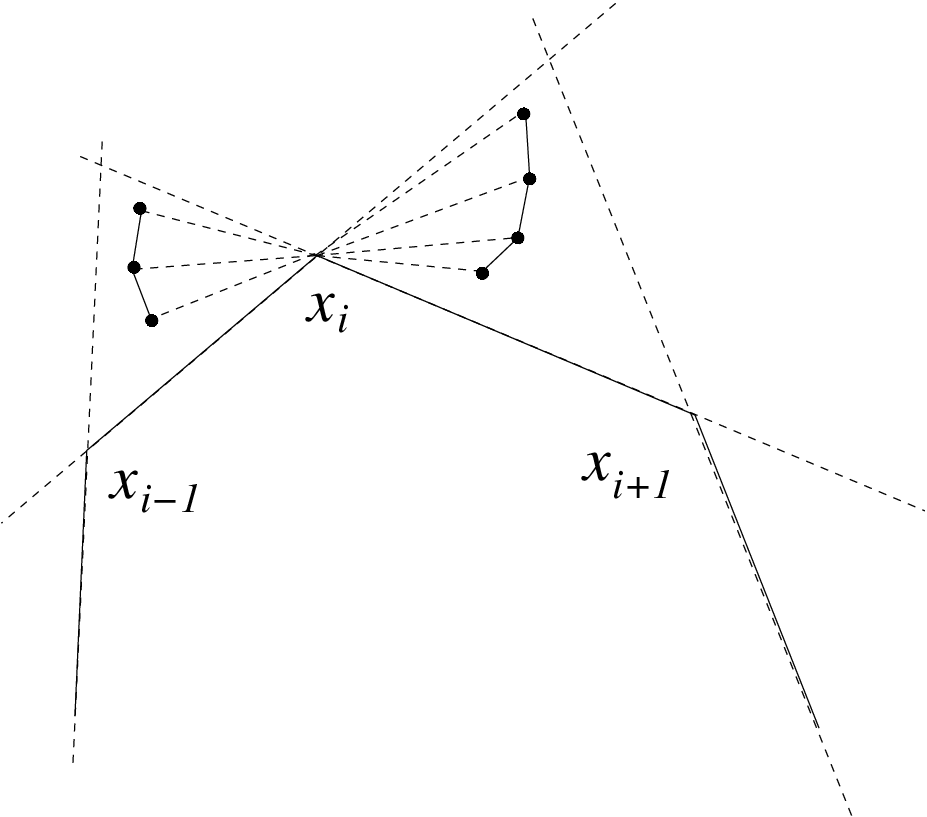}
  \caption{An inner-cap of size 3 with respect to $x_i$ in $P'_{i-1}$ and an inner-cap of size 4 in $P'_{i}$ with respect to $x_i$, which gives 7 points in convex position.}
 \end{center}
\end{figure}

By Observation~\ref{obsadj}, we have $a_i + b_{i + 1} < n$ for all $i$. By applying Theorem~\ref{cupscapsC} with $K = \{x_i\}$, we have
$v_i\le f_{\ell}(a_i+1, b_i+1)< c(\ell+n){a_i+b_i-2 \choose a_i-1}$.  Likewise, by applying Theorem \ref{cupscapsC} with $K = B_i$, we have $h_i\le f_{\ell}(w_i+1, z_i+1)< c(\ell+n){w_i+z_i-2 \choose w_i-1}$.
 Putting everything together, we obtain

\begin{eqnarray*}
       \frac{N^k}{2^{40k^2}}  & \leq & \prod\limits_{i = 2}^{k+1}|P_i| \\\\
     & \leq  & \prod\limits_{i = 2}^{k+1}v_ih_i \\\\
     & \leq &\prod\limits_{i = 2}^{k+1} c^2(\ell + n)^2 \binom{a_i + b_i - 2}{a_i - 1}\binom{w_i + z_i - 2}{w_i - 1}\\\\
     & < & \prod\limits_{i = 2}^{k+1}c^2(\ell + n)^2 2^{a_i + b_i}(2n)^{w_i}\\\\
     & < & (c^2(\ell + n))^{2k} 2^{(k + 1)n + 2n \log (2n)},
\end{eqnarray*}

\noindent where $c$ is the absolute constant from Theorem \ref{cupscapsC}.  Therefore, we have
$$N < c^2(\ell + n)^2 2^{n + 3(n/k)\log (2n) + 40k}.$$

\noindent Since $k = 2\lceil\sqrt{n\log n}\rceil$, this gives us
$$N < \ell^2\cdot  2^{n + O(\sqrt{n\log n})}.$$
Since $|P| = N =  \ell^2 \cdot 2^{n+ C\sqrt{n \log n}} $, by setting $C$ sufficiently large, we have a contradiction.\end{proof}

{\bf Acknowledgements.} This research was initiated during a visit to the American Institute of Mathematics under their SQuaREs program.


\begin{thebibliography}{99}

\bibitem{beck} J. Beck, On the lattice property of the plane and some problems of Dirac, Motzkin and Erd\H os in combinatorial geometry, {\it Combinatorica} \textbf{3} (1983), 281--297. 


\bibitem{chan} T. M. Chan, Optimal partition trees, {\it Discrete Comput. Geom.} \textbf{47} (2012), 661--690.

\bibitem{dil} R. P. Dilworth, A decomposition theorem for partially ordered sets, {\it Ann. of Math.} \textbf{51} (1950), 161--166.


\bibitem{ES35} P. Erd\H os and G. Szekeres, A combinatorial problem in geometry, {\it Compos. Math.} {\bf 2} (1935), 463--470


\bibitem{es2} P. Erd\H os and G. Szekeres, On some extremum problems in elementary geometry, {\it Ann. Univ. Sci. Budapest. E\"otv\"os Sect. Math.} \textbf{3-4} (1960/1961), 53--62.

\bibitem{holm} A. Holmsen, H. N. Mojarrad, J. Pach and G. Tardos, Two extensions of the Erd\H os--Szekeres problem, {\it J. Eur. Math. Soc.} \textbf{22} (2020), 3981--3995

\bibitem{mat} J. Matou\v{s}ek, Efficient partition trees, \emph{Discrete Comput. Geom.} \textbf{8} (1992), 315--334.


\bibitem{MS} G. Moshkovitz and A. Shapira, Ramsey theory, integer partitions and a new proof of the Erd\H os--Szekeres theorem, {\it Adv. Math.} \textbf{262} (2012), 1107--1129.

\bibitem{PS} J. Pach and J. Solymosi, Canonical theorems for convex sets, {\it Discrete Comput. Geom.} {\bf 19} (1998), 427--435. 

\bibitem{PV} A. P\'or and P. Valtr, The partitioned version of the Erd\H os--Szekeres theorem, {\it Discrete Comput. Geom.} {\bf 28} (2002), 625--637.

\bibitem{S}
A. Suk, On the Erd\H os--Szekeres convex polygon problem, {\it J. Amer. Math. Soc.} {\bf 30} (2017), 1047--1053.
\end{thebibliography}
\end{document}